\documentclass[11pt]{amsart}
\usepackage{amsmath}
\usepackage{amscd}
\usepackage{amssymb}
\usepackage{amsthm}
\usepackage{dsfont} 
\usepackage{appendix} 
\usepackage{verbatim}
\usepackage[normalem]{ulem} 
\usepackage[all]{xy} 

\begin{document}


\theoremstyle{plain}
\newtheorem{theorem}{Theorem}[section]
\newtheorem{introtheorem}{Theorem}

\theoremstyle{plain}
\newtheorem{proposition}[theorem]{Proposition}
\newtheorem{prop}[theorem]{Proposition}

\theoremstyle{plain}
\newtheorem{corollary}[theorem]{Corollary}

\theoremstyle{plain}
\newtheorem{lemma}[theorem]{Lemma}

\theoremstyle{plain}
\newtheorem{expectation}[theorem]{Expectation}

\theoremstyle{remark}
\newtheorem{remind}[theorem]{Reminder}

\theoremstyle{definition}
\newtheorem{condition}[theorem]{Condition}

\theoremstyle{definition}
\newtheorem{construction}[theorem]{Construction}

\theoremstyle{definition}
\newtheorem{definition}[theorem]{Definition}

\theoremstyle{definition}
\newtheorem{question}[theorem]{Question}

\theoremstyle{definition}
\newtheorem{example}[theorem]{Example}

\theoremstyle{definition}
\newtheorem{notation}[theorem]{Notation}

\theoremstyle{definition}
\newtheorem{convention}[theorem]{Convention}

\theoremstyle{definition}
\newtheorem{assumption}[theorem]{Assumption}
\newtheorem{isoassumption}[theorem]{Isomorphism Assumption}

\newtheorem{finassumption}[theorem]{Finiteness Assumption}
\newtheorem{indhypothesis}[theorem]{Inductive Hypothesis}

\theoremstyle{remark}
\newtheorem{remark}[theorem]{Remark}

\numberwithin{equation}{subsection}





\newcommand{\mathscr}[1]{\mathcal{#1}}
\newcommand{\TODO}{{\color{red} TODO}}
\newcommand{\CHANGE}{{\color{red} CHANGE}}
\newcommand{\annette}[1]{{\color{red}Annette: #1 }}
\newcommand{\simon}[1]{{\color{green}Simon: #1 }}
\newcommand{\giuseppe}[1]{{\color{blue}Giuseppe: #1 }}

\newcommand{\tensor}{\otimes}
\newcommand{\Fil}{\tn{Fil}}
\newcommand{\Gh}{\mathcal{G}}
\newcommand{\Fh}{\shfF}
\newcommand{\Oh}{\mathcal{O}}
\newcommand{\GrVec}{\textnormal{GrVec}}
\newcommand{\BN}{\textnormal{BN}}
\newcommand{\modulo}{\textnormal{mod}}
\newcommand{\CH}{\mathrm{CH}}
\newcommand{\im}{\mathrm{im}}
\newcommand{\CHM}{\mathrm{CHM}}
\newcommand{\NUM}{\mathrm{NUM}}
\newcommand{\ab}{\mathrm{ab}}
\newcommand{\num}{\mathrm{num}}
\newcommand{\id}{\mathrm{id}}
\newcommand{\isocan}{\xrightarrow{\hspace{1.85pt}\sim \hspace{1.85pt}}}
\newcommand{\isom}{\cong}
\newcommand{\red}{\mathrm{red}}
\newcommand{\Betti}{R_B}
\newcommand{\ladic}{R_\ell}
\newcommand{\Hodge}{R_H}
\newcommand{\MHM}{\mathrm{MHM}}
\newcommand{\Hh}{\mathcal{H}}
\newcommand{\A}{\mathbb{A}}
\newcommand{\gm}{\mathrm{gm}}
\newcommand{\qfh}{\mathrm{qfh}}
\newcommand{\DM}{\tn{\tbf{DM}}}
\newcommand{\DA}{\tn{\tbf{DA}}}
\newcommand{\DMeff}{\tn{\tbf{DM}}^{\eff}}
\newcommand{\DAeff}{\tn{\tbf{DA}}^{\eff}}
\newcommand{\DAgm}{\tn{\tbf{DA}}_{\gm}}
\newcommand{\Bei}{\mathcyr{B}}
\newcommand{\kd}{\tn{kd}}
\newcommand{\Sm}{\tn{\tbf{Sm}}}
\newcommand{\SmCor}{\tn{\tbf{SmCor}}}
\newcommand{\cor}{\tn{\tbf{cor}}}
\newcommand{\Mor}{\textnormal{Mor}}
\newcommand{\Hom}{\textnormal{Hom}}
\newcommand{\End}{\textnormal{End}}
\newcommand{\sss}{\textnormal{ss}}
\newcommand{\Sh}{\tn{\tbf{Sh}}}
\newcommand{\et}{\tn{\'{e}t}}
\newcommand{\an}{\tn{an}}
\newcommand{\D}{\tn{D}}
\newcommand{\eff}{\tn{eff}}
\newcommand{\DMgm}{\DM_{\tn{gm}}}
\newcommand{\vp}{\varphi}
\newcommand{\Sym}{\textnormal{Sym}}
\newcommand{\OSym}{\textnormal{coSym}}
\newcommand{\Mof}{M_1}
\newcommand{\CGS}{\tn{\tbf{cGrp}}}
\newcommand{\tr}{\textnormal{tr}}
\newcommand{\one}{\mathds{1}}
\newcommand{\Spec}{\textnormal{Spec}}
\newcommand{\Sch}{\tn{\tbf{Sch}}}
\newcommand{\Nor}{\tn{\tbf{Nor}}}
\newcommand{\Lie}{\tn{Lie}}
\newcommand{\Ker}{\tn{Ker}}
\newcommand{\Frob}{\tn{Frob}}
\newcommand{\Ver}{\tn{Ver}}
\newcommand{\N}{\mathbb{N}}
\newcommand{\Z}{\mathbb{Z}}
\newcommand{\Q}{\mathbb{Q}}
\newcommand{\Ql}{\mathbb{Q}_{\ell}}
\newcommand{\R}{\mathbb{R}}
\newcommand{\C}{\mathbb{C}}
\newcommand{\G}{\mathbb{G}}

\newcommand{\xra}{\xrightarrow}
\newcommand{\xla}{\xleftarrow}
\newcommand{\sxra}[1]{\xra{#1}}
\newcommand{\sxla}[1]{\xla{#1}}

\newcommand{\sra}[1]{\stackrel{#1}{\ra}}
\newcommand{\sla}[1]{\stackrel{#1}{\la}}
\newcommand{\slra}[1]{\stackrel{#1}{\lra}}
\newcommand{\sllra}[1]{\stackrel{#1}{\llra}}
\newcommand{\slla}[1]{\stackrel{#1}{\lla}}

\newcommand{\ira}{\stackrel{\simeq}{\ra}}
\newcommand{\ila}{\stackrel{\simeq}{\la}}
\newcommand{\ilra}{\stackrel{\simeq}{\lra}}
\newcommand{\illa}{\stackrel{\simeq}{\lla}}

\newcommand{\ul}[1]{\underline{#1}}
\newcommand{\tn}[1]{\textnormal{#1}}
\newcommand{\tbf}[1]{\textbf{#1}}

\newcommand{\Spt}{\mathop{\mathbf{Spt}}\nolimits}
\newcommand{\ra}{\rightarrow}
\newcommand{\old}{\mathrm{old}}

\setcounter{tocdepth}{1}

\title{Numerical functors on Voevodsky's motives}

\author{Giuseppe Ancona}
\address{Universit\"at Zurich, Winterthurerstr. 190, CH-8057 Zurich}
\email{giuseppe.ancona@math.uzh.ch}


\maketitle
\begin{center}
\today
\end{center}


\begin{abstract}
We study mixed versions of  the classical quotient functor from Chow motives to numerical motives. We compare two natural definitions, which turn out to be very different. We investigate fullness, conservativity and exactness of these two functors.
\end{abstract}

\tableofcontents

\section*{Introduction}
In order to understand algebraic cycles over smooth projective $k$-schemes,  the functor
\[\num : \CHM(k) \longrightarrow \NUM(k),\]
from Chow motives to numerical motives (over the base field $k$) has proved to be an important tool, one reason being that, contrary to $\CHM(k)$, the category $\NUM(k)$ is semisimple \cite{Jann}.

Some examples were it has been successfully exploited are the proof of Bloch's conjecture for surfaces dominated by a product of curves \cite{Kim}, the proof that rational and numerical equivalence coincide for products of elliptic curves over finite fields \cite{Kahn}, and the proof of the Bloch-Beilinson conjecture for products of elliptic curves (over some special fields) \cite{JannKim}.

The main classical properties of the functor $\num$ are the following. 
\begin{enumerate}
\item The functor $\num$ is full (by construction),
\item The kernel of $ \num$ is  the biggest proper tensor ideal of $\CHM(k)$ \cite[Lemme 7.1.1]{AK},
\item The functor $\num$ is conservative when restricted to Chow motives arising from abelian varieties \cite{Kim}. 
\end{enumerate}

Keeping the above results in mind, in order to understand algebraic cycles over $k$-schemes that are smooth but not projective, it seems useful to provide a mixed analogue of such a  functor, that would extend $\num$ from $\CHM(k)$ to $\DMgm(k)$, Voevodsky's triangulated category of geometric motives.

There are at least two reasonable candidates for this extension. The first one is the quotient functor
\[p:\DMgm(k) \longrightarrow  \DMgm(k)/\mathcal{N}, \]
where $\mathcal{N}$ is the biggest proper tensor ideal of $\DMgm(k)$. The second one is a triangulated functor
\[\pi:\DMgm(k) \longrightarrow  D^b(\NUM(k)),\]
defined by Bondarko (see Lemma $\ref{unicity}$  and Remark $\ref{rem bon}$) using the fact that a mixed motive is an iterated extension of Chow motives, so one applies $\num$ to each of them.

\

The aim of this article is to compare these two functors, which surprisingly (at least for the author)  turn out to be very different.
\begin{theorem}
Let $\mathcal{T} \subset \DMgm(k)$ be the smallest triangulated category containing finite dimensional Chow motives. 
Then the functor $\pi$ is conservative when restricted to $\mathcal{T}$ (Proposition $\ref{pi conservative}$).

Moreover, when the base field $k$ is not algebraic over a finite field, the following holds :
\begin{enumerate}
\item  The quotient $\DMgm(k)/\mathcal{N}$ has no triangulated structure such that the functor $p:\DMgm(k) \longrightarrow  \DMgm(k)/\mathcal{N}$ is triangulated. (Proposition $\ref{prop not nilp}$), 
\item $\pi$ is not  full, even when restricted to $\mathcal{T}$ (Proposition $\ref{pi not full}$),
\item $p$ is not conservative, even when restricted to $\mathcal{T}$  (Proposition $\ref{ultimo}$).
\item Both (pseudo-abelianisations of) $D^b(\NUM(k))$ and $\DMgm(k)/\mathcal{N}$ are semisimple categories (Remark \ref{basta e avanza}).
\end{enumerate}
\end{theorem}
It is particularly surprisingly for us that, contrary to the pure case, $p$ is not conservative (for this reason we think of $\pi$ as being the good construction for future applications).

\

Recall that the category $\mathcal{T}$ contains the motive of any commutative algebraic group (Remark $\ref{micito}$).
Conjecturally all Chow motives are finite dimensional \cite{Kim}, hence  $\mathcal{T}$ should be the whole $\DMgm(k)$. 

The hypothesis on the base field $k$ in the theorem above is necessary. Indeed, when  $k$ is  algebraic over a finite field, the functor $\num$ is conjectured to be an equivalence of categories. Under this conjecture, both $p$ and $\pi$ are equivalences of categories.

\

There is a piece missing in the complete description of these functors.
\begin{question}
Does the inclusion $\ker F \subset \ker \pi $
hold for any nonzero tensor triangulated functor 
$F:\DMgm(k) \longrightarrow \mathcal{C}$ ?
\end{question}
Note that, if the answer to this question was affirmative, then any realization functor would be conservative (when restricted to $\mathcal{T}$).

\begin{remark}
One of the main conjectures of the theory of motives predicts that there should be a tannakian (non semisimple) category of motives $\mathcal{M}$ inside $\DMgm(k)$ and $\DMgm(k)$ itself should be the derived category of $\mathcal{M}$. Moreover, it is conjectured that the subcategory of semisimple objects $\mathcal{M}^{\sss} \subset \mathcal{M}$ should be related\footnote{More precisely, under the K\"unneth conjecture, one can correct the sign of symmetries in the tensor product to have only even objects and deduce a tannakian category. For details see for example \cite[\S 7]{Milnenum}} to $\NUM(k)$ (but not equivalent: $\NUM(k)$ cannot be tannakian because it has odd objects).

Of course, the two functors $p$ and $\pi$ we consider here do not have any relation with the inclusions $\mathcal{M}\hookrightarrow \DMgm(k)$ for trivial reasons: first because they go in the opposite direction, and second because the categories $\DMgm(k)/\mathcal{N}$ and $D^b(\NUM(k))$ are semisimple.
\end{remark}

\subsection*{Organization of the paper} We start by recalling generalities on rigid categories (Section $\ref{reminder}$). In Section $\ref{setting}$ we introduce the notations and the objects that will be studied in the rest of the paper. The following two sections present the tools we need : finite dimensionality (after Kimura \textit{et al.}) and weight structures (after Bondarko \textit{et al.}). In Section $\ref{numerical}$ we define the functor $\pi$ of the introduction and show that it is a tensor triangulated functor. Finally, the comparison between  the functors $\pi$ and $p$ is carried on in Section $\ref{results}$.

\subsection*{Acknowledgments} I would like to thank Tom Bachmann, Mikhail Bondarko  and Shane Kelly for providing useful explanations. I am also grateful to Javier Fres\'an and Marco Maculan for helpful comments on the exposition.  Finally, I thank the referee for very useful comments.

\section{Reminders on trace, radical and numerical ideal}\label{reminder}
Let $F$ be a commutative $\Q$-algebra. In this section we consider a rigid, symmetric, $F$-linear and pseudo-abelian category $\mathcal{A}$ with unit object $\one$ such that $\End_{\mathcal{A}}(\one)=F\cdot \id.$ \begin{definition}\cite[\S 7.1]{AK}\label{def trace}
Let $f: X \rightarrow X$ be an endomorphism in $\mathcal{A}.$ Let $\eta_f: \one \rightarrow X \otimes X^{\vee}$ be the morphism adjoint to $f$ and $\epsilon:   X \otimes X^{\vee} \rightarrow \one$ be the morphism adjoint to $\id_X$. Define the trace of $f$ to be the unique scalar $\tr(f)\in F$ such that $\tr(f) \cdot \id_{\one}= \epsilon \circ \eta_{f}.$
\end{definition}
\begin{definition}\cite[Lemme 7.1.1]{AK}\label{def numerical ideal}
We say that a morphism $f:X \rightarrow Y$ in $\mathcal{A}$  belongs to the numerical ideal if for all $g: Y \rightarrow X$ the trace of $g\circ f$ vanishes.
\end{definition}
\begin{definition}
An ideal of  $\mathcal{A}$ is a collection of morphisms in $\mathcal{A}$ which is stable by internal sums and external compositions. It is a tensor ideal if it is also stable by external tensor products.
\end{definition}
\begin{proposition}\cite[Proposition 7.1.4(b)]{AK}\label{radical big}
The numerical ideal of $\mathcal{A}$ is a tensor ideal and any other tensor ideal is contained in it.
\end{proposition}
\begin{definition}\cite[D\'efinition 1.4.1]{AK}\label{def radical}
The radical of  $\mathcal{A}$ is the collection of morphisms $f:X \rightarrow Y$ such that, for all $g: Y \rightarrow X$, the morphism $\id - g\circ f$ is invertible.
\end{definition}
\begin{proposition}\cite{Kelly}\label{radical ideal}
The radical of $\mathcal{A}$ is an ideal (which is nontensor in general).
\end{proposition}
Recall that in such a category $\mathcal{A}$, the group of permutation of $n$ elements acts on $X^{\otimes n}$, for any object $X$. In particular one can define the $n$-th symmetric power $\Sym^n X$ and the $n$-th alternating power $\wedge^n X$ of $X$ as direct factors of  $X^{\otimes n}$ via the usual projectors.
\begin{definition}\cite{Kim}\label{def fin dim}
An object $X$ in $\mathcal{A}$ is odd if $\Sym^n X=0$ for some $n$. It is even if $\wedge^n X=0$ for some $n$. It is finite dimensional if it can be decomposed as sum of an odd and an even object.
\end{definition}

\section{Setting}\label{setting} Let us fix  a base field $k$ and a commutative $\Q$-algebra $F$, which will play the role of ring of coefficients. Let $\DMgm(k)_{F}$ be the non-effective category of geometric motives over $k$ with coefficients in $F$  \cite[\S 2]{TMF}.  It is a triangulated, rigid, symmetric and pseudo-abelian $F$-linear category with unit object $\one$ such that $\End_{\DMgm(k)_{F}}(\one)=F\cdot \id$. It canonically contains the category $\CHM(k)_F$ of non-effective Chow motives over $k$ with $F$ coefficients \cite[Proposition 2.1.4]{TMF}.
\begin{definition}\label{def category}
Define $\mathcal{C} \subseteq \CHM(k)_F$ to be the full subcategory of finite dimensional motives. Write $\mathcal{T} \subseteq \DMgm(k)_F$ for the triangulated category generated by $\mathcal{C}$.
\end{definition}
By Kimura \cite[Corollary 5.11 and Proposition 6.9]{Kim},  $\mathcal{C}$ is a  rigid, symmetric and pseudo-abelian $F$-linear category. We will see that $\mathcal{T}$ is a triangulated, rigid, symmetric and pseudo-abelian $F$-linear category. 

\begin{definition}\label{def pure numerical}
Write $\mathcal{N}$ for the numerical ideal (Definition $\ref{def numerical ideal}$) of $\mathcal{T}$ and define the functor $\num: \mathcal{C} \rightarrow \overline{\mathcal{C}}$ to be the  quotient of $\mathcal{C}$ by its numerical ideal (which is the restriction of $\mathcal{N}$ to $\mathcal{C}$).
\end{definition}
In our motivic setting the functor $\num$ is nothing else than the projection to the category of numerical motives \cite[Exemples 2.4(1)]{AndBour}.
\begin{definition}\label{def mixed numerical}
We define $\overline{\mathcal{T}}$ to be the bounded derived category of $\overline{\mathcal{C}}$. We will call it the triangulated category of numerical motives.
\end{definition}
\begin{remark}\label{basta e avanza}
By a theorem of Jannsen \cite{Jann}, $\overline{\mathcal{C}}$ is abelian semisimple, hence the category $\overline{\mathcal{T}}$ is equivalent to $\overline{\mathcal{C}}^{\oplus \Z}$, and, in particular, it is also semisimple.
Note that, as $\overline{\mathcal{C}}$ is a  rigid, symmetric and pseudo-abelian $F$-linear category, $\overline{\mathcal{T}}$ is a triangulated, rigid, symmetric and pseudo-abelian $F$-linear category. 

Note also that the same proof of Jannsen applies to show that the pseudo-abelian envelope of $\mathcal{T}/\mathcal{N}$ is semisimple\footnote{As well as the pseudo-abelian envelope of the quotient of $\DMgm(k)_{F}$ by its numerical ideal, apply for example \cite[Proposition 2.6]{AndBour}  to a realization functor}.
\end{remark}

\section{Finite dimensional motives}\label{kimura}
Along this section, $\mathcal{C}$ is the category of finite dimensional Chow motives over $k$ with coefficients in $F$ (Definitions $\ref{def fin dim}$ and $\ref{def category}$) and $\overline{\mathcal{C}}$ is its quotient as in Definition $\ref{def pure numerical}$. We recall results due to Kimura, O'Sullivan and Jannsen. 
\begin{proposition}\label{prop kimura}\cite[Corollary 5.11 and Proposition 6.9]{Kim}
The category $\mathcal{C}$ is tensor, rigid, additive and pseudo-abelian.
\end{proposition}
\begin{theorem}\label{thm kimura}\cite[Corollary 3.4]{JannKim}
For any $X$ in $\mathcal{C}$ the ideal $\mathcal{N}(X,X)\subset \End_{\mathcal{C}}(X)$ of endomorphisms belonging to the numerical ideal (Definition $\ref{def numerical ideal}$) is nilpotent.
\end{theorem}
\begin{corollary}\label{cor kimura}\cite[Corollary 7.8]{Kim} All projectors in $\overline{\mathcal{C}}$ can be lifted to projectors in $\mathcal{C}$.
\end{corollary}
\begin{theorem}\cite{Jann}
The category $\overline{\mathcal{C}}$ is tensor, rigid  and semi-simple.
\end{theorem}
Note that, because of Proposition $\ref{prop kimura}$ and Corollary $\ref{cor kimura}$, the quotient category $\overline{\mathcal{C}}$ is already pseudo-abelian.
\begin{proposition}\label{prop rad num}\cite[Th\'eor\`eme 8.2.2]{AK}
The radical of $\mathcal{C}$ (Definition $\ref{def radical}$) coincides with the numerical ideal. 
\end{proposition}

\section{Weights in motives}\label{bondarko}
We keep notations from Section $\ref{setting}$, in particular $\mathcal{C}$ and $\mathcal{T}$ are the categories of motives (over $k$ with coefficients in $F$) of Definition $\ref{def category}$. We summarize results due to Bondarko and Wildeshaus in the following theorem.
\begin{theorem}\label{thm weight structure}
The category $\mathcal{T}$ is pseudo-abelian. Moreover, for each pair of integers $a\leq b$ there exists a full pseudo-abelian subcategory
\[\mathcal{T}_{a,b} \subset \mathcal{T}\]
called  the category of motives with motivic weights between $a$ and $b$
(when $a=b$ we shall write $\mathcal{T}_{=a}$ for $ \mathcal{T}_{a,a}$ and call it the category of pure motives of weight $a$), such that the collection of categories $\{\mathcal{T}_{a,b} \}$ satisfies : 
\begin{enumerate}
\item The inclusion $\mathcal{T}_{c,d} \subset \mathcal{T}_{a,b}$, for all integers $a\leq c$ and $b \geq d$,
\item The equality $\mathcal{T}_{a,b}\cap \mathcal{T}_{c,d}= \mathcal{T}_{a,d}$, for all integers $a\geq c$ and $b \geq d$,
\item The equality $\mathcal{T}_{a,b}[n]= \mathcal{T}_{a+n,b+n}$,
\item\label{heart} The equality $\mathcal{T}_{=0}= \mathcal{C}$,
\item The inclusion $\mathcal{T}_{c,d} \otimes \mathcal{T}_{a,b} \subset \mathcal{T}_{a+c,b+d}$,
\item For any $M \in \mathcal{T}$ there exist $a\leq b$ such that $M \in \mathcal{T}_{a,b}$,
\item\label{stability} For any triangle $A\rightarrow B\rightarrow C$ in $\mathcal{T}$,
\[A, C \in \mathcal{T}_{a,b} \Rightarrow B \in \mathcal{T}_{a,b},\]
\item When $b < c$, $\Hom_{\mathcal{T}}(\mathcal{T}_{a,b}  , \mathcal{T}_{c,d} )=0,$
\item\label{min filt} For any $a\leq b < c$ and $X \in \mathcal{T}_{a,c}$ there exists a (nonunique) triangle
\[X_{a,b} \longrightarrow X \longrightarrow X_{b+1,c} \stackrel{\delta}{\longrightarrow}X_{a,b}[1],\]
with $X_{a,b}\in \mathcal{T}_{a,b}$, $X_{b+1,c}\in \mathcal{T}_{b+1,c}$ and $\delta$ in  the radical of $\mathcal{T}$ (Definition $\ref{def radical})$,
\item\label{functoriality min filt} Let $a\leq b < c$ and let $f:X \rightarrow Y$ be a morphism in $\mathcal{T}_{a,c}$.  For any choice of triangles as in $(\ref{min filt})$ $f$ extends to a  morphism of triangles

 \[\xymatrix{
 X_{a,b} \ar[d]^{f_{a,b}} \ar[r]  & X  \ar[d]^{f} \ar[r]& X_{b+1,c} \ar[d]^{f_{b+1,c}}\\
 Y_{a,b}\ar[r]& Y \ar[r] & Y_{b+1,c} } 
\]
(not in a unique way),
\item\label{minimality} Let $a\leq b < c$ and $X \in \mathcal{T}_{a,c}$. Any triangle of the form
\[X'_{a,b} \longrightarrow X \longrightarrow X'_{b+1,c} \stackrel{+1}{\longrightarrow} ,\]
with $X'_{a,b}\in \mathcal{T}_{a,b}$, $X'_{b+1,c}\in \mathcal{T}_{b+1,c}$ is isomorphic to a triangle of the form
\[X_{a,b}\oplus X_b \longrightarrow X \longrightarrow X_{b+1,c}\oplus X_b[1]\stackrel{\delta \oplus \id}{\longrightarrow}X_{a,b}[1]\oplus X_b[1],\]
 for some $X_b \in \mathcal{T}_{b}$ (with $X_{a,b}, X_{b+1,c}$ and $\delta$ as in $(\ref{min filt}))$.
\end{enumerate}
\end{theorem}
These results due to Bondarko \cite[Theorem 4.3.2 and Lemma 5.2.1]{Bon} and  valid on the whole $ \DMgm(k)_F $, with exception of parts (9) and (11), proved in \cite[Proposition 1.2 and Corollary 1.3]{Wild} by Wildeshaus. Although Wildeshaus works with the subcategory $ \DMgm^{\ab}(k)_F \subset \mathcal{T}$ of mixed motives of abelian type \cite[Definition 1.1(c)]{Wild}, the only thing used in his proof is finite dimensionality.
\begin{definition}\label{def length}
The length of a motive $M$ is  the smallest natural number $b-a$ such that $M \in \mathcal{T}_{a,b}$. \end{definition}
The length measures how far a motive is from being pure, for instance Chow motives are of length zero and the motive of an affine curve is (in general) of length one.

\section{Triangulated numerical motives}\label{numerical}
In this section we construct a functor $\pi : \mathcal{T} \rightarrow \overline{\mathcal{T}}$ extending the functor $\num: \mathcal{C} \rightarrow \overline{\mathcal{C}}$ (see Definitions $\ref{def category}$ and $\ref{def mixed numerical}$). We show that it is a tensor and triangulated functor.
Proofs will use the categories $\mathcal{T}_{a,b}$, from Theorem $\ref{thm weight structure}$, and the notion of length of motives  introduced in Definition $\ref{def length}$.
\begin{lemma}\label{beginning}
Let $\pi : \mathcal{T} \rightarrow \overline{\mathcal{T}}$ be any (additve, $F$-linear) triangulated functor, extending 
 the functor $\num: \mathcal{C} \rightarrow \overline{\mathcal{C}}$. Then it satisfies :
 \begin{enumerate}
 \item\label{unox}
 For any morphism of triangles
\[\xymatrix{
 X_{=a} \ar[d]^{f_{=a}} \ar[r]  & X  \ar[d]^{f} \ar[r]& X_{a+1,b} \ar[d]^{f_{a+1,b}}\\
 Y_{=a}\ar[r]& Y \ar[r] & Y_{a+1,b}  }
\]
as in Theorem $\ref{thm weight structure}(\ref{functoriality min filt})$,  the morphsim \[\pi(f): \pi(X) \rightarrow \pi(Y)\] coincides with 
\[\pi(f_{=a})\oplus \pi(f_{a+1,b}) : \pi(X_{=a}) \oplus \pi(X_{a+1,b}) \rightarrow \pi(Y_{=a}) \oplus \pi(Y_{a+1,b}),\]
 \item\label{duex}
 For any  morphism in the radical of $\mathcal{T}$ (Definition $\ref{def radical}$)
 \[\delta : X_{a+1,b} \rightarrow X_{=a}[1],\] 
 with $X_{a+1,b} \in \mathcal{T}_{a+1,b}$ and   $X_{=a} \in \mathcal{T}_{=a}$, one has \[\pi(\delta)=0.\]
 \end{enumerate}
\end{lemma}

\begin{proof}
Let us start showing $(\ref{duex})$. We first write a morphism of triangles as in Theorem $\ref{thm weight structure}(\ref{functoriality min filt})$, 
\[\xymatrix{
 X_{=a+1} \ar[d]^{\delta_{a+1}} \ar[r]  & X_{a+1,b}   \ar[d]^{\delta} \ar[r]& X_{a+2,b}  \ar[d]\\
 X_{=a}[1]\ar[r]^{\id} & X_{=a}[1]\ar[r] & 0  }
\]
and deduce we are reduced to show that $\pi(\delta_{a+1})=0$. Note now that $\delta_{a+1}$ belongs to the radical because the radical is an ideal (Proposition $\ref{radical ideal}$) and $\delta$ belongs to it.  After shifting we can suppose that $\delta_{a+1}$ is a morphism in $\mathcal{C}$. By Proposition $\ref{prop rad num}$, $\delta_{a+1}$  belongs to the numerical ideal $\mathcal{N}$. As $\pi$ restricted to  $\mathcal{C}$ coincides with the functor $\num$, we conclude
$\pi(\delta_{a+1})=0$.

Let us turn to $(\ref{unox})$.  By  $(\ref{duex})$ we have  $ \pi(X)=\pi(X_{=a}) \oplus \pi(X_{a+1,b})$ and $ \pi(Y)=\pi(Y_{=a}) \oplus \pi(Y_{a+1,b})$, hence it is enough to show that 
\[\End(\pi(X_{a+1,b}), \pi(Y_{=a}) ) = 0 .\]
This follows from the equivalence $\overline{\mathcal{T}}\cong\overline{\mathcal{C}}^{\oplus \Z}$ of Remark $\ref{basta e avanza}$.
\end{proof}

\begin{lemma}\label{unicity}
There exists a unique additive, $F$-linear functor $\pi : \mathcal{T} \rightarrow \overline{\mathcal{T}}$ commuting with shifts, extending the functor $\num: \mathcal{C} \rightarrow \overline{\mathcal{C}}$ and verifying the property $(\ref{unox})$ from Lemma $\ref{beginning}$.
\end{lemma}
\begin{proof}
 We show the existence and unicity of $\pi$ on the different subcategories  $\mathcal{T}_{a,b}$ by induction on the natural number $l=b-a$. 

When $b=a$ we can suppose after shifting that $a=b=0$. As $\mathcal{T}_{=0}= \mathcal{C}$ by Theorem $\ref{thm weight structure}(\ref{heart})$, the functor $\pi$ coincides there with $\num.$

Let $f: X \rightarrow Y$ be a morphism in $\mathcal{T}_{a,b}$. Write a morphism of triangles as in Theorem $\ref{thm weight structure}(\ref{functoriality min filt})$, 
\[\xymatrix{
 X_{=a} \ar[d]^{f_{=a}} \ar[r]  & X  \ar[d]^{f} \ar[r]& X_{a+1,b} \ar[d]^{f_{a+1,b}}\\
 Y_{=a}\ar[r]& Y \ar[r] & Y_{a+1,b}  }
\]
and note that we know by induction that  the morphisms $ \pi(f_{=a}) $ and $\pi(f_{a+1,b}) $ are uniquely determined. Then property 
$(\ref{unox})$ from Lemma $\ref{beginning}$ forces the value of $ \pi(f) $. This shows the unicity of $\pi$.

\

To show the existence it is enough to prove that the maps
$\pi(f_{=a})$ and $\pi(f_{a+1,a+l})$ vanish, for any morphism of triangles of the shape
\[\xymatrix{
 X_{=a} \ar[d]^{f_{=a}} \ar[r]  & X  \ar[d]^{0} \ar[r]& X_{a+1,a+l} \ar[d]^{f_{a+1,a+l}}\\
 Y_{=a}\ar[r]& Y \ar[r] & Y_{a+1,a+l}  }
\]
(note that the maps $\pi(f_{=a})$ and $\pi(f_{a+1,a+l})$ are well determined by induction).

Applying the functor $\Hom(X_{=a} , \cdot )$ to the second triangle in the above diagram, one constructs a  factorisation
\[\xymatrix{
  {}  &  X_{=a}\ar[d]^{f_{=a}}\ar[ld]  \\
 Y_{a+1,a+l}[-1]\ar[r]^{\delta}  & Y_{=a} }
\]
which implies that $f_{=a}$ belongs to the radical (Definition $\ref{def radical}$), as the radical is an ideal (Proposition $\ref{radical ideal}$) and $\delta$ belongs to it. After shifting we can suppose that $f_{=a}$ is a morphism in $\mathcal{C}$. By Proposition $\ref{prop rad num}$, $f_{=a}$  belongs to the numerical ideal $\mathcal{N}$. As $\pi$ restricted to  $\mathcal{C}$ coincides with the functor $\num$, we conclude
$\pi(f_{=a})=0$ .

Similarly, there is a factorization
\[\xymatrix{
  X_{a+1,a+l} \ar[d]_{f_{a+1,a+l}} \ar[r]^{\delta'}  &  X_{=a}[1]\ar[ld] \\
 Y_{a+1,a+l} & {} }
\]
with $\delta'$ in the radical. By functoriality it is enough to show that $\pi(\delta')=0$. We write
\[\xymatrix{
 X_{=a+1} \ar[d]^{\delta'_{a+1}} \ar[r]  & X_{a+1,a+l}   \ar[d]^{\delta'} \ar[r]& X_{a+2,a+l}  \ar[d]\\
 X_{=a}[1]\ar[r]^{\id} & X_{=a}[1]\ar[r] & 0  }
\]
hence  $\pi(\delta')=\pi(\delta'_{a+1})$ by property 
$(\ref{unox})$ from Lemma $\ref{beginning}$.  We argue as above to conclude that $\delta'_{a+1}\in \mathcal{N}$, hence $\pi(\delta'_{a+1})=0.$
\end{proof}
The rest of the section is devoted to show that the unique  functor $\pi$ from Lemma $\ref{unicity}$ is a tensor triangulated functor.
\begin{remark}\label{rem bon}
\begin{enumerate}
\item If one wants to work with a motivic category which is larger than $\mathcal{T}$ one could consider a similar construction due to Bondarko \cite[Remark 6.3.2(3)]{Bon2}, who defines a triangulated functor
\[\DMgm(k)_F \longrightarrow D^b(\NUM(k)_F)\]
(and even a functor  $\DMgm(k)_F \longrightarrow K^b(\CHM(k)_F)$).
Its restriction to $\mathcal{T}$ coincides with our $\pi$ because of the unicity part of Lemma $\ref{unicity}$.

The functor considered by Bondarko was not known to be a tensor functor. This has been recently established by Bachmann \cite[Lemma 18]{Bach}. The rest of the section can be seen as an alternative proof of these facts in our simplified context.

\item In the next section we will show that $\pi$ is conservative on $\mathcal{T}$. Extending this result to the whole $\DMgm(k)_F$ without assuming finite dimensionality seems out of reach with the present technology.
\end{enumerate}
\end{remark}

\begin{proposition}\label{prop pi triang} The unique functor $\pi$ from Lemma $\ref{unicity}$ is  triangulated.
\end{proposition}
\begin{proof}
Let us fix a triangle
\[A \longrightarrow  B \longrightarrow C\] 
in $\mathcal{T}$. We want to show that the functor $\pi$ induces a triangle in $\overline{\mathcal{T}}$, or equivalently a long exact sequence in the semisimple category $\overline{\mathcal{C}}$.

After shifting we can suppose that $A,B,C \in \mathcal{T}_{0,l}$. Write a triangle \[A_0 \longrightarrow  A \longrightarrow A_{1,l}\] as in Theorem $\ref{thm weight structure}(\ref{min filt})$, and similarly for $B$ and $C$. 

\

Step (i). By Theorem $\ref{thm weight structure}(\ref{functoriality min filt})$ we get maps
\[A_0 \longrightarrow B_0 \longrightarrow C_0\]
(note that this is not a triangle in general). We claim that $C_0$ is a direct factor of $B_0$ (and the map above is the induced projection). 

Indeed, consider the commutative diagram 
\[\xymatrix{
B_{0} \ar[d] \ar[r]^{\id}  & B_0  \ar[d] \ar[r]& 0 \ar[d]\\
 B \ar[d]  \ar[r]  & C  \ar[d] \ar[r]& A[1]\ar[d]^{\id}\\
 B_{1,l}\ar[r] & X \ar[r] & A[1]}
\]
where both vertical and horizontal lines are triangles. By looking at the bottom horizontal triangle we deduce that $X\in\mathcal{T}_{1,l}$, using Theorem  $\ref{thm weight structure}(\ref{stability})$. We get the claim by Theorem $\ref{thm weight structure}(\ref{minimality})$ applied to the second vertical arrow.

\

Step (ii). The functor $\pi$ in degree zero applied to the triangle $A \rightarrow  B \rightarrow C$ reads
\[\pi(A_0) \longrightarrow \pi(B'_0) \oplus \pi(C_0) \longrightarrow \pi(C_0)\]
for some $B'_0\in \mathcal{C}$. This is a complex (as $\pi$ is a functor) and the last arrow is surjective by Step (i). Let us show that it is also exact in the middle. 

First, consider the commutative diagram 
\[\xymatrix{
0 \ar[d] \ar[r]  & B_{1,l}[-1]   \ar[d]^{\delta} \ar[r]^{\id} & B_{1,l}[-1]  \ar[d]^{\delta'} \\
 C_0 \ar[d]^{\id}   \ar[r]  & B_0 \ar[d] \ar[r]& B'_0 \ar[d] \\
 C_0 \ar[r] & B \ar[r] & B'}
\]
where both vertical and horizontal lines are triangles. As $\delta$ belongs to the radical, which is an ideal (Proposition $\ref{radical ideal}$), so it does $\delta'$. In particular the triangle
\[B'_0 \longrightarrow B' \longrightarrow B_{1,l}\]
is of the shape as in Theorem $\ref{thm weight structure}(\ref{min filt})$. 

Second, consider the commutative diagram
\[\xymatrix{
0 \ar[d] \ar[r]  & C_0  \ar[d] \ar[r]^{\id} & C_0 \ar[d] \\
 A \ar[d]^{\id}   \ar[r]  & B \ar[d]  \ar[r]& C \ar[d] \\
 A \ar[r] & B' \ar[r] & C_{1,l} }
\]
of triangles. Applying the argument of step (i) to the triangle 
\[ C_{1,l} [-1] \longrightarrow A \longrightarrow B',\]
we deduce that  $B'_0$ is a direct factor of $A_0$. In particular, the sequence we constructed
\[\pi(A_0) \longrightarrow \pi(B'_0) \oplus \pi(C_0) \longrightarrow \pi(C_0)\]
is exact not only on the right but also in the middle, and it can be rewritten
\[\pi(A'_0)  \oplus \pi(B'_0) \longrightarrow \pi(B'_0) \oplus \pi(C_0) \longrightarrow \pi(C_0) \longrightarrow 0.\]

\

Step (iii). By continuing this procedure we get the long exact sequence we wanted. Note also that by passing from the triangle
$A \rightarrow  B \rightarrow C$ to the triangle $C_{1,l} [-1] \rightarrow A \rightarrow B'$ one of the elements has length strictly smaller hence the procedure stops after finitely many steps.
\end{proof}
\begin{lemma}

The category $\mathcal{T} \subseteq \DMgm(k)_F$ is stable by tensor products, duals and twists.
\end{lemma}
\begin{proof}
Let us show the statement for tensor product (for duals and twists the argument is analogue). We show that $A\otimes B \in \mathcal{T}$ for any $A,B \in \mathcal{T}$ by induction on the sum of the lengths of $A$ and $B$. 

When the sum is zero we are done as $\mathcal{C}$ is a tensor category. For a higher sum, we can suppose (without loss of generality) that $A$ has positive length. Then write a triangle
\[A_a  \longrightarrow A  \longrightarrow A_{a+1,b} \]
as in Theorem $\ref{thm weight structure}(\ref{min filt})$. From this we deduce a triangle
\[A_a \otimes B \longrightarrow A\otimes B  \longrightarrow A_{a+1,b}\otimes B. \]
As the first and the third terms are in $\mathcal{T}$ (by induction) so is the middle one.
\end{proof}

\begin{remark}\label{micito}
The same proof shows that the category $\DMgm^{\ab}(k)_F$ of mixed motives of abelian type \cite[Definition 1.1(c)]{Wild} is tensor rigid. It contains in particular the tensor category generated by $1$-motives, hence the motives of commutative algebraic groups \cite{AEH}.
\end{remark}

\begin{proposition} The functor $\pi$ from Lemma $\ref{unicity}$ commutes with tensor products, symmetries and duals.
\end{proposition}
\begin{proof}
We show the statement for the tensor product (for symmetries and duals the argument is analogue). Given two morphisms $f:A \to C$ and $g:B \to D$ in $\mathcal{T}$ we show that $\pi(A\otimes B)=\pi(A)\otimes \pi(B)$, $\pi(C\otimes D)=\pi(C)\otimes \pi(D)$ and $\pi(f\otimes g)=\pi(f)\otimes \pi(g)$. 

We can suppose  $A,C\in\mathcal{T}_{a,a+l}$ and $B,D\in\mathcal{T}_{b,b+m}$. We argue by induction on  the sum  $l+m$. When this sum is zero we are done as this reduces to the fact that $\pi$ is a tensor functor on $\mathcal{C}$. 

For a positive sum, we can suppose (without loss of generality) that $l>0$. Then write a triangle
\[A_a  \longrightarrow A  \longrightarrow A_{a+1,a+l} \stackrel{\delta}{\longrightarrow}A_{a}[1], \]
as in Theorem $\ref{thm weight structure}(\ref{min filt})$. From this we deduce a triangle
\[A_a \otimes B \longrightarrow A\otimes B  \longrightarrow A_{a+1,a+l}\otimes B \stackrel{\delta\otimes \id}{\longrightarrow}(A_{a}\otimes B) [1]. \]
By Lemma $\ref{beginning}(\ref{duex})$, we have $\pi(\delta)=0$. By induction hypothesis $\pi(\delta\otimes \id)=\pi(\delta)\otimes \pi(\id)$, hence it also vanishes. 

As the functor $\pi$ is triangulated the two triangles above give
\[\pi(A)=   \pi(A_{a+1,a+l}) \oplus  \pi(A_{a}) \,\,,\,\, \pi(A\otimes B) =  \pi(A_{a+1,a+l}\otimes B) \oplus \pi (A_{a}\otimes B)  . \]
Putting these two equalities together and using the induction hypothesis we conclude that $\pi(A\otimes B)=\pi(A)\otimes \pi(B)$.  The argument for $\pi(C\otimes D)=\pi(C)\otimes \pi(D)$ and $\pi(f\otimes g)=\pi(f)\otimes \pi(g)$ is analogue.
\end{proof}

\section{Results}\label{results}
In this section we compare the functor $\pi : \mathcal{T} \rightarrow \overline{\mathcal{T}}$ from Lemma $\ref{unicity}$ with the quotient functor \[p:\mathcal{T} \rightarrow \mathcal{T} / \mathcal{N}  \] (see Definition $\ref{def pure numerical}$). Again in this section we will use the categories $\mathcal{T}_{a,b}$, from Theorem $\ref{thm weight structure}$, and the notion of length of motives (Definition $\ref{def length}$).

\begin{proposition}\label{pi not full} The functor $\pi$ is not full, provided that the base field $k$ is not algebraic over a finite field.
\end{proposition}
\begin{proof}
Consider a nonzero morphism $n: X \rightarrow Y$ in $\mathcal{C}$ belonging to the $\mathcal{N}$ (when $k$ is not algebraic over a finite field there are plenty of nonzero cycles that are numerically trivial, such as degree zero divisors over smooth projective curves of nonzero genus). Take a motive $C$ sitting in the triangle
\[C \stackrel{g}\longrightarrow  X \stackrel{n}{\longrightarrow} Y.\]
Note that, by Lemma $\ref{beginning}(\ref{unox})$, we have $\pi(C)=\pi(X)\oplus \pi(Y)[-1]$. 

Define the morphism ${\alpha : \pi(X) \rightarrow  \pi(C)}$ as the direct sum of the identity map on $\pi(X) $ and the zero map to $\pi(Y)[-1]$. We claim that $\alpha$ is not the image any morphism $f: X \rightarrow C$. If so, one would have a triangle
\[\xymatrix{
 X \ar[d]^{f} \ar[rd]^{\id_X+r}  &  {} \\
 C \ar[r]^{g}  & X  }
\]
with $r \in \mathcal{N}$. This means that $\id_X+r$ belongs to the image of the composition by $g$,
\[ g \circ \cdot : \Hom(X,C) \longrightarrow \Hom(X,X) \]
that is, to the kernel of the composition by $n$
\[ n\circ \cdot\ : \Hom(X,X) \longrightarrow \Hom(X,Y) .\]
In other words, $n \circ (\id_X+r) =0$.

On the other hand, by Corollary $\ref{cor kimura}$, $\id_X+r$ is invertible, hence $n=0$, which leads to a contradiction.
\end{proof}
\begin{remark}
When the base field $k$ is algebraic over a finite field it is conjectured that the projection $\num:\mathcal{C} \rightarrow \overline{\mathcal{C}}$ (Definition $\ref{def pure numerical}$) is an equivalence of categories. Assuming this conjecture one can show by induction on length that $\pi : \mathcal{T} \rightarrow \overline{\mathcal{T}} $ is also an equivalence.
\end{remark}
\begin{proposition}\label{prop not nilp}
Suppose that the base field $k$ is not algebraic over a finite field. Then there exists a motive $X\in \mathcal{T}_{0,1}$ and an endomorphism $f:X \rightarrow X$ belonging to $\mathcal{N}$ such that for any diagram
\[\xymatrix{
 X_{=0} \ar[d]^{f_{0}} \ar[r]  & X  \ar[d]^{f} \ar[r]& X_{=1} \ar[d]^{f_{1}}\\
 X_{=0}\ar[r]& X \ar[r] & X_{=1}  }
\]
as in Theorem $\ref{thm weight structure}(\ref{functoriality min filt})$, the maps $f_0$ and $f_1$ do not belong to $\mathcal{N}$. In particular:
\begin{enumerate}
\item The tensor ideal $\ker \pi$ is strictly contained in $\mathcal{N}$,
\item\label{no triang} The quotient $\mathcal{T}/\mathcal{N}$ has no triangulated structure such that the functor $\mathcal{T} \rightarrow \mathcal{T}/\mathcal{N}$ is triangulated.
\end{enumerate}
\end{proposition}
\begin{proof}
Suppose we have a map $f$ as in the statement. Then :
\begin{enumerate}
\item By Proposition $\ref{radical big}$, $\mathcal{N}$ contains $\ker \pi$. On the other hand, by construction, $f$ does not belong to $\ker \pi$ but it belongs to $\mathcal{N}$,
\item If the quotient $\mathcal{T}/\mathcal{N}$ was  triangulated, then $p(X)=p(X_{=0})\oplus p(X_{=1})$ and $f$ would be a zero map which is the sum of two nonzero maps.  
\end{enumerate}
To construct such an $f$ consider first a nonzero map $\alpha : M \rightarrow N$ in  $\mathcal{C}$, such that $\End(M)$ and $\End(N)$ are of dimension $1$ and there are no nonzero map from $N$ to $M$ (which implies that $\alpha$ is in  $\mathcal{N}$). This can be found by taking $M=\one$, $N$ to be the motivic $H^1$ of an elliptic curve without complex multiplication and the map to be the difference of the origin and a  point which is not torsion.

Now consider the endomorphism of triangles

\[\xymatrix{
 M \oplus M \ar[d]^{\left(\begin{smallmatrix}0 & 1 \\ 0 & 0 \end{smallmatrix}\right)} \ar[r]^{\left(\begin{smallmatrix} \alpha & \alpha \\ 0 & \alpha \end{smallmatrix}\right)} & N \oplus N  \ar[d]^{\left(\begin{smallmatrix}0 & 1 \\ 0 & 0 \end{smallmatrix}\right)} \ar[r]& X \ar[d]^{f}\\
 M\oplus M \ar[r] & N \oplus N \ar[r] & X }
\]
 and note that the first two vertical maps are not in $\mathcal{N}$, hence it is enough to show that for any $g: X \rightarrow X$ we have $\tr(gf)=0$.
 
 By Theorem $\ref{thm weight structure}(\ref{functoriality min filt})$, such an endomorphism $g$ is equivalent to a commutative square
 \[\xymatrix{
 M \oplus M \ar[d]^{\left(\begin{smallmatrix}a & b \\ c & d \end{smallmatrix}\right)} \ar[r] & N \oplus N  \ar[d]^{\left(\begin{smallmatrix}a' & b' \\ c' & d' \end{smallmatrix}\right)} \\
 M\oplus M \ar[r] & N \oplus N }
\]
 with the horizontal maps given by $\left(\begin{smallmatrix} \alpha & \alpha \\ 0 & \alpha \end{smallmatrix}\right)$. The condition of commutativity implies $c=c'=0$.
 
The relation
 \[\tr\left(\left(\begin{matrix}a & b \\ 0 & d \end{matrix}\right)\left(\begin{matrix}0 & 1 \\ 0 & 0 \end{matrix}\right)\right) = \tr\left(\left(\begin{matrix}a' & b' \\ 0 & d' \end{matrix}\right)\left(\begin{matrix}0 & 1 \\ 0 & 0 \end{matrix}\right)\right) = 0,\]
 forces $\tr(gf)=0$. Indeed, given an endomorphism of triangles, if the traces of two arrows are zero, then so is the third (this can be checked for example after realization).
\end{proof}

\begin{proposition}\label{pi conservative} Let $X\in \mathcal{T}$ be a motive. The ideal \[\{f:X \to X, \pi(f)=0\} \lhd \End(X)\]
is nilpotent. In particular the functor $\pi$ is conservative.
\end{proposition}
\begin{proof}
By \cite[7.2.8]{AK} it is enough to show that there exists a positive integer $N_X$ (depending only on X) such that for all endomorphisms ${f:X \rightarrow X}$ beloging to $\ker \pi $ one has $f^{N_X}=0.$

We show this fact by induction on the length. When the length is zero this is Proposition $\ref{thm kimura}$. For a higher length write an endomorphism of triangles
\[\xymatrix{
 X_{=a} \ar[d]_{f_{=a}} \ar[r]  & X  \ar[d]^{f} \ar[r] & X_{a+1,b} \ar[d]^{f_{a+1,b}}\\
 X_{=a} \ar[r] & X \ar[r] & X_{a+1,b} }
\]
as in Theorem $\ref{thm weight structure}(\ref{functoriality min filt})$. 
For $m=\max\{N_{X_{a+1,b}}, N_{X_{a}}\}$ one has

\[\xymatrix{
 X_{=a} \ar[d]_{(f_{=a})^m=0} \ar[r]  & X  \ar[d]^{f^m} \ar[r] & X_{a+1,b} \ar[d]^{(f_{a+1,b})^m=0}\\
 X_{=a} \ar[r] & X \ar[r] & X_{a+1,b} }
\]
hence the square of $f^m$ vanishes. This means that $N_X=2m$ suffices (and note that it does depend only on $X$).

To deduce the conservativity of $\pi$, notice that, as $\pi$ is triangulated, it is enough to show that $\pi(X)=0$ implies
$X=0$. But,  $\pi(X)=0$ implies  that $\id_X$ is in the kernel of $\pi$, hence $0=\id_X^{N_X}=\id_X$.
\end{proof}

\begin{proposition}\label{ultimo}  Suppose that the base field $k$ is not algebraic over a finite field.  Then there exists an endomorphism $f: X \rightarrow X$ in $\mathcal{T}$ belonging to the numerical ideal $\mathcal{N}$ but which is not nilpotent. Moreover, the functor $p$ is not conservative.  
\end{proposition}
\begin{proof}
First consider a nonzero map $\alpha : M \rightarrow N$ in  $\mathcal{C}$, such that $\End(M)$ and $\End(N)$ are of dimension $1$ and there are no nonzero map from $N$ to $M$ as in the proof of Proposition $\ref{prop not nilp}$.

Now consider the endomorphism of triangles

\[\xymatrix{
 M \oplus N \ar[d]^{\left(\begin{smallmatrix} 1 & 0 \\ 0 & 0 \end{smallmatrix}\right)} \ar[r]^{\left(\begin{smallmatrix} 0 & 0 \\ \alpha  & 0 \end{smallmatrix}\right)} & 
 M \oplus N  \ar[d]^{\left(\begin{smallmatrix}0 & 0 \\ 0 & 1 \end{smallmatrix}\right)} \ar[r]& X \ar[d]^{f}\\
 M\oplus 
 N \ar[r] & M \oplus N \ar[r] & X }
\]
 and note that $f^n=f$ for all positive integers $n$. Note also that $f$ is  nonzero otherwise one would have a commutative triangle
 \[\xymatrix{
 {}  &  M \oplus N \ar[d]^{\left(\begin{smallmatrix}0 & 0 \\ 0 & 1 \end{smallmatrix}\right)} \ar[ld]_{\left(\begin{smallmatrix} a & 0 \\ c & d \end{smallmatrix}\right)} \\
 M \oplus N \ar[r]_{\left(\begin{smallmatrix} 0 & 0 \\ \alpha  & 0 \end{smallmatrix}\right)}  & M \oplus N  }
\]
which is impossible.

We have to show that for any $g: X \rightarrow X$ we have $\tr(gf)=0$.
 
 By Theorem $\ref{thm weight structure}(\ref{functoriality min filt})$, such an endomorphism $g$ is equivalent to a commutative square
 \[\xymatrix{
 M \oplus N \ar[d]^{\left(\begin{smallmatrix}a & 0 \\ c & d \end{smallmatrix}\right)} \ar[r] & M \oplus N  \ar[d]^{\left(\begin{smallmatrix}a' & 0 \\ c & d' \end{smallmatrix}\right)} \\
 M\oplus N \ar[r] & M \oplus N }
\]
 with the horizontal maps given by $\left(\begin{smallmatrix} 0 & 0 \\ \alpha  & 0 \end{smallmatrix}\right)$. The condition of commutativity implies $a=d'$. We deduce the relation
 \[\tr\left(\left(\begin{matrix}a & 0 \\ c & d \end{matrix}\right)\left(\begin{matrix}1 & 0 \\ 0 & 0 \end{matrix}\right)\right) = \tr\left(\left(\begin{matrix}a' & 0\\ c' & d' \end{matrix}\right)\left(\begin{matrix}0 & 0 \\ 0 & 1 \end{matrix}\right)\right) .\]
 This implies $\tr(gf)=0$ (as one can check after realization).

\

To show that $p$ is not conservative consider $h=f-\id$. As we have shown that $f$ belongs to $\mathcal{N}$ then $p(h)$ is invertible. Now, if $h$ was invertible, then by applying the functor $\pi$ we deduce that also the endomorphism
\[\pi\left(\begin{matrix}0 & 0 \\ 0 & -1 \end{matrix}\right) \,\, \textrm{and} \,\, \pi\left(\begin{matrix}-1 & 0 \\ 0 & 0 \end{matrix}\right)\]
of $\pi(M\oplus N)$ would be invertible. This gives a contradiction as $\pi$ is conservative (Proposition $\ref{pi conservative}$).
\end{proof}

\bibliographystyle{alpha}	
\def\cprime{$'$}

\end{document}